\documentclass{amsart}
\usepackage{amssymb}
\usepackage{amsmath,amsthm,amsfonts,amssymb,latexsym,mathrsfs,color,hyperref}

\newtheorem{theorem}{Theorem}
\newtheorem{corollary}[theorem]{Corollary}
\newtheorem{proposition}[theorem]{Proposition}

\newtheorem{lemma}[theorem]{Lemma}

\newcommand{\lrmin}{{\rm lrmin\,}}

\newcommand{\cpkk}{{\rm cpk\,}}

\newcommand{\cpk}{{\rm cpk\,}}

\newcommand{\lrmax}{{\rm lrmax\,}}
\newcommand{\lramax}{{\rm lramax\,}}
\newcommand{\Lramax}{{\rm Lramax\,}}

\newcommand{\ap}{{\rm ap\,}}

\newcommand{\lpk}{{\rm lpk\,}}

\newcommand{\des}{{\rm des\,}}

\newcommand{\exc}{{\rm exc\,}}

\newcommand{\aexc}{{\rm aexc\,}}
\newcommand{\we}{{\rm wexc\,}}

\newcommand{\cyc}{{\rm cyc\,}}

\newcommand{\fix}{{\rm fix\,}}

\newcommand{\mdn}{\mathcal{D}}

\newcommand{\msn}{\mathfrak{S}_n}

\newcommand{\ms}{\mathfrak{S}}

\newcommand{\lrf}[1]{\lfloor #1\rfloor}

\newcommand{\mqn}{\mathcal{Q}_n}
\newcommand{\z}{ \mathbb{Z}}
\newcommand{\asc}{{\rm asc\,}}

\numberwithin{equation}{section}

\begin{document}
\title{The $1/k$-Eulerian polynomials of type $B$}

\author{Shi-Mei Ma}
\address{School of Mathematics and Statistics,
Northeastern University at Qinhuangdao, Hebei, P.R. China}
\email{shimeimapapers@163.com}
\thanks{Supported by NSFC (Grant 11401083,11571235) and NSC (107-2115-M-001-009-MY3).}
\author{Jun Ma}
\address{Department of Mathematics, Shanghai Jiao Tong University, Shanghai, P.R. China}
\email{majun904@sjtu.edu.cn}
\author{Jean Yeh}
\address{Departemnt of Mathematics, National Kaohsiung Normal University, Kaohsiung 82446, Taiwan}
\email{chunchenyeh@nknu.edu.tw}
\author{Yeong-Nan Yeh}
\address{Institute of Mathematics,
        Academia Sinica, Taipei, Taiwan}
\email{mayeh@math.sinica.edu.tw}

\subjclass[2010]{Primary 05A05, 05E45; Secondary 05A15, 26C05}


\keywords{Eulerian polynomials, Gamma-positivity, Derangement polynomials}

\begin{abstract}
In this paper, we give a type $B$ analogue of the $1/k$-Eulerian polynomials. Properties of this kind of polynomials,
including combinatorial interpretations, recurrence relations and $\gamma$-positivity are studied.
In particular, we show that the $1/k$-Eulerian polynomials of type $B$ are $\gamma$-positive when $k>0$.
Moreover, we obtain the corresponding results for derangements of type $B$.
We show that a type $B$ $1/k$-derangement polynomials $d_n^B(x;k)$ are bi-$\gamma$-positive when $k\geq 1/2$.
In particular, we get a symmetric decomposition of $d_n^B(x;1/2)$
in terms of the classical derangement polynomials.
\end{abstract}

\maketitle
\section{Introduction}
In this paper, we first give a type $B$ analogue of the $1/k$-Eulerian polynomials and then we
derive the corresponding results for derangements of type $B$. In the following, we give a survey on the study of $1/k$-Eulerian polynomials.

Throughout this paper, we always let $k$ be a fixed positive number. Following Savage and Viswanathan~\cite{Savage12},
the {\it $1/k$-Eulerian polynomials} $A_{n}^{(k)}(x)$ are defined by
\begin{equation}\label{Ankx-deff}
\sum_{n=0}^\infty A_{n}^{(k)}(x)\frac{z^n}{n!}=\left(\frac{1-x}{e^{kz{(x-1)}}-x}\right)^{\frac{1}{k}}.
\end{equation}
When $k=1$, the polynomial $A_n^{(k)}(x)$ reduces to the classical Eulerian polynomial $A_n(x)$.
Savage and Viswanathan~\cite{Savage12} showed that
\begin{equation*}\label{Ankx-Savage}
A_n^{(k)}(x)=\sum_{\textbf{e}\in I_{n,k}}x^{\asc(\textbf{e})},
\end{equation*}
where $I_{n,k}=\left\{ \textbf{e}~|~0\leq e_i\leq (i-1)k\right\}$ is the set of $n$-dimensional \emph{$k$-inversion sequences}
with $\textbf{e}=(e_1,e_2,\ldots,e_n)\in\z^n$ and
$$\asc(\textbf{e})=\#\left\{i:1\leq i\leq n-1~\big{|}~\frac{e_i}{(i-1)k+1}<\frac{e_{i+1}}{ik+1}\right\}.$$

A $k$-{\it Stirling permutation} of order $n$ is a permutation of the multiset
$\{1^k,2^k,\ldots,n^k\}$ such that for each $i$, $1\leq i\leq n$,
all entries between the two occurrences of $i$ are at least $i$. When $k=2$, the $k$-Stirling permutation
reduces to the ordinary Stirling permutation~\cite{Gessel78}.
Denote by $\mqn(k)$ the set of $k$-{\it Stirling permutation} of order $n$.
Let $\sigma=\sigma_1\sigma_2\cdots\sigma_{kn}\in \mqn(k)$.
We say that an index $i\in \{2,3,\ldots,nk-k+1\}$ is a {\it longest ascent plateau} if $\sigma_{i-1}<\sigma_i=\sigma_{i+1}=\cdots=\sigma_{i+k-1}$.
Let $\ap(\sigma)$ be the number of the longest ascent plateaus of $\sigma$.
According to~\cite[Theorem~2]{Ma15}, we have
\begin{equation*}\label{Ankx-stirling}
A_n^{(k)}(x)=\sum_{\sigma\in \mqn(k)}x^{\ap(\sigma)}.
\end{equation*}

Let $[n]=\{1,2,\ldots,n\}$.
Let $\msn$ denote the symmetric group of all permutations of $[n]$ and let $\pi=\pi(1)\cdots\pi(n)\in\msn$.
A {\it descent} (resp.~{\it ascent, excedance}) of $\pi$ is an index $i\in[n-1]$
such that $\pi(i)>\pi(i+1)$ (resp.~$\pi(i)<\pi(i+1)$, $\pi(i)>i$,). Let $\des(\pi)$ (resp.~$\asc(\pi)$, $\exc(\pi)$,) denote the number of descents (resp.~ascents, excedances) of $\pi$.
It is well known that the statistics $\des(\pi),\asc(\pi)$ and $\exc(\pi)$ are equidistributed.
The {\it Eulerian polynomial} $A_n(x)$ is given as follows:
$$A_n(x)=\sum_{\pi\in\msn}x^{\des(\pi)}=\sum_{\pi\in\msn}x^{\asc(\pi)}=\sum_{\pi\in\msn}x^{\exc(\pi)}.$$
In~\cite{Foata70},
Foata and Sch\"utzenberger introduced a $q$-analog of $A_n(x)$ defined by
\begin{equation*}
A_n(x,q)=\sum_{\pi\in\msn}x^{\exc(\pi)}q^{\cyc(\pi)},
\end{equation*}
where $\cyc(\pi)$ is the number of cycles of $\pi$.
Brenti~\cite{Brenti00} showed that some of the crucial properties of Eulerian polynomials have nice $q$-analogues for the polynomials $A_{n}(x;q)$ .
In particular, we have $$\sum_{n=0}^\infty A_n(x,q)\frac{z^n}{n!}=\left(\frac{1-x}{e^{z{(x-1)}}-x}\right)^q.$$
Therefore, $A_{n}^{(k)}(x)=k^n A_{n}(x,1/k)$.
A bijective proof of the following identity was recently given in~\cite{Chao19}: $$\sum_{\textbf{e}\in I_{n,k}}x^{\asc(\textbf{e})}=\sum_{\pi\in\msn}x^{\asc(\pi)}k^{n-\lrmin(\pi)}.$$

Let $f(x)=\sum_{i=0}^nf_ix^i$ be a symmetric polynomial, i.e., $f_i=f_{n-i}$ for any $0\leq i\leq n$. Then $f(x)$ can be expanded uniquely as
$f(x)=\sum_{k=0}^{\lrf{\frac{n}{2}}}\gamma_kx^k(1+x)^{n-2k}$, and it is said to be {\it $\gamma$-positive}
if $\gamma_k\geq 0$ for $0\leq k\leq \lrf{\frac{n}{2}}$ (see~\cite{Gal05}).
The $\gamma$-positivity of $f(x)$ implies symmetry and unimodality of $f(x)$.
We refer the reader to Athanasiadis's survey article~\cite{Athanasiadis17} for details.
The $\gamma$-positivity of Eulerian polynomials was first observed by
Foata and Sch\"utzenberger~\cite{Foata70}. Subsequently, Foata and Strehl~\cite{Foata73} proved the $\gamma$-positivity of Eulerian polynomials
by using the well known Foata-Strehl action.
By using the theory of enriched $P$-partitions, Stembridge~\cite[Remark 4.8]{Stembridge97} showed that
\begin{equation*}
A_n(x)=\frac{1}{2^{n-1}}\sum_{i=0}^{(n-1)/2}4^iP(n,i)x^i(1+x)^{n-1-2i},
\end{equation*}
where $P(n,i)$ is the the number of permutations in $\msn$ with $i$ {\it interior peaks}, i.e., the indices $i\in\{2,\ldots,n-1\}$ such that $\pi(i-1)<\pi(i)>\pi(i+1)$.
It should be noted that the polynomial $A_{n}^{(k)}(x)$ is not symmetric if $k\neq 1$, and so it is not $\gamma$-positive.

Let $\pm[n]=[n]\cup\{-1,\ldots,-n\}$.
Let $B_n$ be the {\it hyperoctahedral group} of rank $n$. Let $w=w(1)w(2)\cdots w(n)\in B_n$.
Elements of $B_n$ are signed permutations of $\pm[n]$ with the property that $w(-i)=-w(i)$ for all $i\in [n]$.
Let $$\des_B(w)=\#\{i\in\{0,1,\ldots,n-1\}\mid w(i)>w({i+1})\},$$ where $w(0)=0$.
We say that $i\in [n]$ is a {\it weak excedance} of $w$ if $w(i)=i$ or $w(|w(i)|)>w(i)$ (see~\cite[p.~431]{Brenti94}).
Let $\we(w)$ be the number of weak excedances of $w$.
According to~\cite[Theorem~3.15]{Brenti94}, the statistics $\des_B(w)$ and $\we(w)$ have the same distribution over $B_n$,
and their common enumerative polynomial is the
{\it Eulerian polynomial of type $B$}:
$$B_n(x)=\sum_{w\in B_n}x^{\des_B(w)}=\sum_{w\in B_n}x^{\we(w)}.$$


A {\it left peak} of $\pi\in\msn$ is an index $i\in[n-1]$ such that $\pi(i-1)<\pi(i)>\pi(i+1)$, where we take $\pi(0)=0$.
Let $\lpk(\pi)$ be the number of left peaks in $\pi$.
Let $Q(n,i)$ be the number of permutations in $\msn$ with $i$ left peaks.
By using the theory of enriched $P$-partitions, Petersen~\cite[Proposition~4.15]{Petersen07} obtained the following result.
\begin{theorem}\label{PetersenThm}
We have
$B_n(x)=\sum_{i=0}^{\lrf{n/2}}4^iQ(n,i)x^i(1+x)^{n-2i}$.
\end{theorem}

In recent years, ordinary and $q$-generalizations of Theorem~\ref{PetersenThm} have been studied by several authors, see~\cite{Lin15,Zeng16,Zhuang17} and references therein.

We now introduce a {\it type $B$ $1/k$-Eulerian polynomials} $B_n^{(k)}(x)$ by using the following generating function
\begin{equation}\label{Bnkx-deff}
\sum_{n=0}^\infty B_n^{(k)}(x)\frac{z^n}{n!}=\left(\frac{(1-x)e^{kz(1-x)}}{1-xe^{2kz(1-x)}}\right)^{\frac{1}{k}}.
\end{equation}
When $k=1$, the polynomial $B_n^{(k)}(x)$ reduces to $B_n(x)$ (see~\cite[Eq.~(14)]{Brenti94}).
Comparing~\eqref{Bnkx-deff} with~\eqref{Ankx-deff}, we see that
$$B_n^{(k)}(x)=\sum_{i=0}^n\binom{n}{i}(2x)^iA_{i}^{(k)}(1/x)(1-x)^{n-i}.$$

This paper is organized as follows.
As a refinement of Theorem~\ref{PetersenThm}, in the next section we show that the polynomials $B_n^{(k)}(x)$ are $\gamma$-positive when $k$ is positive.
In Section~\ref{Section03}, we study a type $B$ $1/k$-derangement polynomials $d_n^B(x;k)$. In particular, in Theorem~\ref{theorem-derange},
we find that $d_n^B(x;1/2)$ is the ascent polynomial of permutations in $\ms_{n+1}$ with no successions.

%
%
%
\section{The $1/k$-Eulerian polynomials of type $B$}\label{Section02}
An element is a {\it left-to-right maximum} of $\pi\in\msn$ if it is larger than or equal to all the elements to its left.
We always assume that $\pi(1)$ is a left-to-right maximum. Let $\lrmax(\pi)$ be the number of left-to-right maxima of $\pi$.
We can write $\pi\in\msn$ in standard cycle decomposition, where each cycle is written with its largest entry first and the cycles are written in increasing order of their largest entry.
A {\it cycle peak} of $\pi$ is an index $i$ such that $\pi^{-1}(i)<i>\pi(i)$. Let $\cpkk(\pi)$ be the number of cycle peaks of $\pi$.
By using the {\it fundamental transformation} of Foata and Sch\"utzenberge~\cite{Foata70}, it is easy to verify that
\begin{equation*}\label{cpklpk}
\sum_{\pi\in\msn}x^{\cpk(\pi)}y^{\cyc(\pi)}=\sum_{\pi\in\msn}x^{\lpk(\pi)}y^{\lrmax(\pi)}.
\end{equation*}

We can now conclude the first main result of this paper.
\begin{theorem}\label{thm1}
For $n\geq 1$, we have
\begin{equation*}
B_n^{(k)}(x)=\sum_{w\in B_n}x^{\we(w)}k^{n-\cyc(w)},
\end{equation*}
and the polynomials $B_n^{(k)}(x)$ satisfy the recurrence relation
\begin{equation}\label{Bnkx-recu}
B_{n+1}^{(k)}(x)=(1+x+2knx)B_n^{(k)}(x)+2kx(1-x)\frac{d}{dx}B_n^{(k)}(x),
\end{equation}
with the initial conditions $B_0^{(k)}(x)=1$ and $B_1^{(k)}(x)=1+x$.
When $k>0$, the polynomial $B_n^{(k)}(x)$ is $\gamma$-positive. More precisely, we have
\begin{equation}\label{Bnk-gamma}
B_n^{(k)}(x)=\sum_{i=0}^{\lrf{n/2}}\left(\sum_{j=i}^{n-1}b(n,i,j)k^j4^i\right)x^i(1+x)^{n-2i},
\end{equation}
where the numbers $b(n,i,j)$ satisfy the recurrence relation
\begin{equation}\label{bnij}
b(n+1,i,j)=b(n,i,j)+2ib(n,i,j-1)+(n-2i+2)b(n,i-1,j-1),
\end{equation}
with $b(1,0,0)=1$ and $b(1,i,j)=0$ for $(i,j)\neq (0,0)$.
Let $$b_n(x,q)=\sum_{i=0}^{\lrf{n/2}}\sum_{j=i}^{n-1}b(n,i,j)x^iq^j.$$
Then
\begin{equation}\label{bnxq-comb}
b_n(x,q)=\sum_{\pi\in\msn}x^{\cpk(\pi)}q^{n-\cyc(\pi)},
\end{equation} and
the generating function of the polynomials $b_n(x,q)$ is given as follows:
$$b(x,q,z)=\sum_{n=0}^\infty b_n(x,q)\frac{z^n}{n!}=\left(\frac{\sqrt{1-x}}{\sqrt{1-x}\cosh(qz\sqrt{1-x})-\sinh(qz\sqrt{1-x})}\right)^{1/q}.$$
\end{theorem}

Let $w\in B_n$.
An element $w(i)$ is a {\it left-to-right absolute maximum} of $w$ if $|w(i)|>|w(j)|$ for every $j\in [i-1]$ or $i=1$.
Let $\Lramax(w)$ denote the set of left-to-right absolute maxima of $w$, and let $\lramax(w)=\#\Lramax(w)$.
As usual, we denote by $\overline{i}$ the negative element $-i$.
For example, $\lramax\{\overline{2}~\overline{4}15\overline{3}\}=|\{\overline{2},\overline{4},5\}|=3$.
Let $\asc_B(w)=\#\{i\in\{0,1,\ldots,n-1\}\mid w(i)<w({i+1})\}$, where $w(0)=0$.
We can write $w$ by its standard cycle decomposition, in which each cycle has its largest (in absolute value) element first and the cycles are written in increasing order of the absolute value of their first elements. Let $\cyc(w)$ be the number of cycles of $w$.
Let $\tau(w)$ be word obtained by deleting all parentheses in the standard cycle decomposition of $w$. It is clear that $\tau$ is a bijection.
From the proof of~\cite[Theorem~3.15]{Brenti94}, we see that
\begin{equation*}
\sum_{w\in B_n}x^{\we(w)}y^{\cyc(w)}=\sum_{w\in B_n}x^{\asc_B(w)}y^{\lramax(w)}.
\end{equation*}
Therefore, another combinatorial interpretation of $B_n^{(k)}(x)$ is given as follows:
\begin{equation*}
B_n^{(k)}(x)=\sum_{w\in B_n}x^{\asc_B(w)}k^{n-\lramax(w)}.
\end{equation*}

In the rest of this section, we shall prove Theorem~\ref{thm1}.
Let $V$ be an alphabet whose letters are regarded as independent commutative
indeterminates.
A {\it context-free grammar} $G$  over $V$ is a set of substitution rules replacing a variable in $V$ by
a Laurent polynomial of variables in $V$,
see~\cite{Chen17,Ma19} for details.
The formal derivative $D:=D_G$ with respect to $G$ is defined as a linear operator acting on Laurent polynomials with variables in $V$
such that each substitution rule is treated as the common differential rule that satisfies the relations. We have $D(c)=0$ for a constant $c$,
$D(u+v)=D(u)+D(v)$ and $D(uv)=D(u)v+uD(v)$.
Following~\cite{Chen17}, a {\it grammatical labeling} is an assignment of the underlying elements of a combinatorial structure
with variables, which is consistent with the substitution rules of a grammar.

In the following discussion, we always write $w\in B_n$ by its standard cycle decomposition.
For $w\in B_n$, we say that $i\in [n]$ is an {\it anti-excedance} of $w$ if $w(i)=\overline{i}$ or $w(|w(i)|)<w(i)$. Let
Let $\aexc(w)$ be the number of anti-excedances of $w$.
It is clear that $\we(w)+\aexc(w)=n$ for $w\in B_n$.
The following lemma is fundamental.
\begin{lemma}\label{lemma-Bnkx}
Let $G=\{I\rightarrow I(x+y),~x\rightarrow 2kxy,~y\rightarrow 2kxy\}$.
We have
\begin{equation}\label{DnI-grammar}
D_G^n(I)=I\sum_{w\in B_n}x^{\we(w)}y^{\aexc(w)}k^{n-\cyc(w)}.
\end{equation}
\end{lemma}
\begin{proof}
We first introduce a grammatical labeling of $w\in B_{n}$ as follows:
\begin{itemize}
  \item [\rm ($L_1$)]If $i$ is a weak excedance, then put a superscript label $x$ right after $i$;
 \item [\rm ($L_2$)]If $i$ is an anti-excedance, then put a superscript label $y$ right after $i$;
\item [\rm ($L_3$)] Put a subscript label $k$ just before every element of $w$ except the first element in each cycle;
\item [\rm ($L_4$)]Put a subscript label $I$ right after $w$.
\end{itemize}
The weight of $w$ is the product of its labels.
Note that the weight of $w$ is given by $$Ix^{\exc(w)}y^{\aexc(w)}k^{n-\cyc(w)}.$$
Every permutation in $B_n$ can be obtained from a permutation in $B_{n-1}$ by inserting $n$ or $\overline{n}$.
For $n=1$, we have $B_1=\{(1^z)_I,(\overline{1}^y)_I\}$.
Note that $D_G(I)=I(x+y)$.
Then the sum of weights of the elements in $B_1$ is given by $D_G(I)$.
Hence the result holds for $n=1$.
We proceed by induction on $n$.
Suppose we get all labeled permutations in $w\in B_{n-1}$, where $n\geq 2$. Let
$\widetilde{w}$ be obtained from $w\in B_{n-1}$ by inserting $n$ or $\overline{n}$.
When the inserted $n$ or $\overline{n}$ forms a new cycle, the insertion corresponds to the substitution rule $I\rightarrow I(x+y)$.
If $i$ is a weak excedance of $w$, the changes of labeling are illustrated as follows:
$$\cdots(i^x)\cdots\mapsto \cdots(n^y_ki^x)\cdots;~~\cdots(i^x)\cdots\mapsto \cdots(\overline{n}^x_ki^y)\cdots;$$
$$\cdots(\cdots w(i)^x_k w(i+1)\cdots w(j))\cdots\mapsto \cdots(n^y_k w(i+1)\cdots w(j)_k\cdots w(i)^x)\cdots;$$
$$\cdots(\cdots w(i)^x_k w(i+1)\cdots w(j))\cdots\mapsto \cdots(\overline{n}^x_k w(i+1)\cdots w(j)_k\cdots w(i)^y)\cdots.$$
If $i$ is an anti-excedance of $w$, the changes of labeling are illustrated as follows:
$$\cdots(\overline{i}^y)\cdots\mapsto \cdots(n^y_k\overline{i}^x)\cdots;~~\cdots(\overline{i}^y)\cdots\mapsto \cdots(\overline{n}^x_k\overline{i}^y)\cdots;$$
$$\cdots(\cdots w(i)^y_k w(i+1)\cdots w(j))\cdots\mapsto \cdots(n^y_k w(i+1)\cdots w(j)_k\cdots w(i)^x)\cdots;$$
$$\cdots(\cdots w(i)^y_k w(i+1)\cdots w(j))\cdots\mapsto \cdots(\overline{n}^x_k w(i+1)\cdots w(j)_k\cdots w(i)^y)\cdots.$$
In each case, the insertion of $n$ or $\overline{n}$ corresponds to the substitution rule $x\rightarrow 2kxy$ or $y\rightarrow 2kxy$.
it is routine to check that the action of $D_G$ on elements of $w\in B_{n-1}$ generates all elements of $\widetilde{w}\in B_n$.
\end{proof}

Let $$F_{n}(x,y;k)=\sum_{w\in B_n}x^{\we(w)}y^{\aexc(w)}k^{n-\cyc(w)}.$$
\begin{lemma}\label{lemma-Bnkx-EGF}
We have
$$F(x,y;k)=\sum_{n=0}^\infty F_{n}(x,y;k)\frac{z^n}{n!}=\left(\frac{(y-x)e^{kz(y-x)}}{y-xe^{2kz(y-x)}}\right)^{1/k}.$$
\end{lemma}
\begin{proof}
By using Lemma~\ref{lemma-Bnkx}, we get $D_G^{n}(I)=IF_{n}(x,y;k)$.
It follows from $D_G^{n+1}(I)=D_G(IF_{n}(x,y;k))$ that
\begin{equation}\label{recu-Fnkxy}
F_{n+1}(x,y;k)=(x+y)F_{n}(x,y;k)+2kxy\left(\frac{\partial}{\partial x}+\frac{\partial}{\partial y}\right)F_{n}(x,y;k),
\end{equation}
with the initial conditions $F_{0}(x,y;k)=1$ and $F_{1}(x,y;k)=x+y$.
By rewriting~\eqref{recu-Fnkxy} in terms of generating function $F:=F(x,y;k)$, we have
\begin{equation}\label{recu-Fnkxy02}
\frac{\partial}{\partial z}F=(x+y)F+2kxy\left(\frac{\partial}{\partial x}+\frac{\partial}{\partial y}\right)F.
\end{equation}
It is routine to check that the generating function
$$\widetilde{F}=\left(\frac{(y-x)e^{kz(y-x)}}{y-xe^{2kz(y-x)}}\right)^{1/k}$$
satisfies~\eqref{recu-Fnkxy02}. Also, this generating function gives $\widetilde{F}(0,0;k)=1$. Hence $F=\widetilde{F}$.
\end{proof}

\noindent{\bf A proof of Theorem~\ref{thm1}:}
\begin{proof}
Combining~\eqref{Bnkx-deff} and Lemma~\ref{lemma-Bnkx-EGF}, we obtain
\begin{equation}\label{FxyzBxyz}
F_{n}(x,y;k)=y^nB_n^{(k)}\left(\frac{x}{y}\right).
\end{equation}
Therefore,
\begin{equation*}
B_n^{(k)}(x)=\sum_{w\in B_n}x^{\we(w)}k^{n-\cyc(w)}.
\end{equation*}
Combining~\eqref{recu-Fnkxy} and~\eqref{FxyzBxyz}, it is routine to verify~\eqref{Bnkx-recu}.
We now consider a change of the grammar $G$ given in Lemma~\ref{lemma-Bnkx}.
Setting $u=x+y$ and $v=xy$, we get
$D(I)=Iu,~D(u)=4kv$ and $D(v)=2kuv$. If $V=\{I,u,v\}$ and
\begin{equation}\label{G1grammar}
G_1=\{I\rightarrow Iu,~u\rightarrow 4kv,~v\rightarrow 2kuv\},
\end{equation}
Note that $D_{G_1}(I)=Iu,~D_{G_1}^2(I)=I(u^2+4kv),~
D_{G_1}^3(I)=I(u^3+(12k+8k^2)uv)$ and
$D_{G_1}^4(I)=I(u^4+(24k+32k^2+16k^3)u^2v+(48k^2+32k^3)v^2)$.
By induction, it is routine to verify that
$$D_{G_1}^n(I)=I\sum_{i=0}^{\lrf{n/2}}\sum_{j=i}^{n-1}b(n,i,j)k^j4^iv^iu^{n-2i}.$$
Then upon taking $u=x+y$ and $v=xy$, we get~\eqref{Bnk-gamma}.
Note that
\begin{align*}
D_{G_1}^{n+1}(I)&=D_{G_1}\left(I\sum_{i=0}^{\lrf{n/2}}\sum_{j=i}^{n-1}b(n,i,j)k^j4^iv^iu^{n-2i}\right)\\
&=\sum_{i,j}b(n,i,j)k^j4^iv^i\left(u^{n-2i+1}+2kiu^{n-2i+1}+4k(n-2i)vu^{n-2i-1}\right).
\end{align*}
Equating the coefficients of $k^j4^iv^iu^{n+1-2i}$ in both sides of the above expression,
we get~\eqref{bnij}. Multiplying both sides of~\eqref{bnij} by $x^iq^j$, we obtain
\begin{equation}\label{bnxq-recu}
b_{n+1}(x,q)=(1+nqx)b_n(x,q)+2qx(1-x)\frac{\partial}{\partial x}b_n(x,q).
\end{equation}
In particular, $b_0(x,q)=b_1(x,q)=1$.

Assume that~\eqref{bnxq-comb} holds.
Given $\sigma\in\msn$.
Let $\sigma_i$ be an element of $\ms_{n+1}$ obtained from $\sigma$ by inserting the entry
$n+1$ right after $i$ if $i\in [n]$ or as a new cycle $(n+1)$ if $i=n+1$.
It is clear that
$$ \cyc(\sigma_i)=\left\{
              \begin{array}{ll}
                \cyc(\sigma), & \hbox{if $i\in [n]$;} \\
                \cyc(\sigma)+1, & \hbox{if $i=n+1$.}
              \end{array}
            \right.
$$
Therefore, we have
\begin{align*}
&{b}_{n+1}(x,q)\\
&=\Sigma_{\pi\in\ms_{n+1}}x^{\cpk(\pi)}q^{n+1-\cyc(\pi)}\\
  &=\Sigma_{i=1}^{n+1}\Sigma_{\sigma\in\msn}x^{\cpk(\sigma_i)}q^{n+1-\cyc(\sigma_i)}\\
  &=\Sigma_{\sigma\in\msn}x^{\cpk(\sigma)}q^{n-\cyc(\sigma)}+
\Sigma_{i=1}^{n}\Sigma_{\sigma\in\msn}x^{\cpk(\sigma_i)}q^{n+1-\cyc(\sigma)}\\
  &={b}_{n}(x,q)+\Sigma_{\sigma\in\msn}\left(2\cpk(\sigma)x^{\cpk(\sigma)}+(n-2\cpk(\sigma))x^{\cpk(\sigma)+1}\right)q^{n+1-\cyc(\sigma)}\\
  &={b}_{n}(x,q)+nqx{b}_{n}(x,q)+2q(1-x)\Sigma_{\pi\in\msn}\cpk(\sigma)x^{\cpk(\sigma)}q^{n-\cyc(\sigma)},
\end{align*}
and~\eqref{bnxq-recu} follows, as desired.
Set $b=b(x,q,z)$.
By rewritten~\eqref{bnxq-recu} in terms of $b$, we have
\begin{equation}\label{Q-recu}
(1-qxz)\frac{\partial b}{\partial z}=b+2qx(1-x)\frac{\partial b}{\partial x}.
\end{equation}
It is routine to verify that
$$\widetilde{b}(x,q,z)=\left(\frac{\sqrt{1-x}}{\sqrt{1-x}\cosh(qz\sqrt{1-x})-\sinh(qz\sqrt{1-x})}\right)^{1/q}$$
satisfies~\eqref{Q-recu}. Also, this generating function gives $\widetilde{b}(x,q,0)=1$ and $\widetilde{b}(0,q,z)=e^{z}$.
Hence $\widetilde{b}(x,q,z)=b(x,q,z)$.
This completes the proof.
\end{proof}

When $q=1$, the generating function $b(x,q,z)$ reduces to the the generating function of the polynomials $\sum_{i=0}^{\lrf{n/2}}Q(n,i)x^i$,
which is due to Gessel~\cite[A008971]{Sloane}. Thus the polynomial $b_n(x,q)$ can be called the {\it $1/q$-left peak polynomial}.
From the explicit formula of $b(x,q,z)$, it is routine to verify the following result.
\begin{corollary}
For $n\geq 1$, we have $$\sum_{\pi\in\msn}x^{\cpk(\pi)}(-1)^{n-\cyc(\pi)}=(1-x)^{\lrf{n/2}}.$$
\end{corollary}

Let $$T(n,i)=\frac{n!}{i!(n-2i)!2^i}$$
be the {\it Bessel number}, which is the number of involutions of $[n]$ with $i$ pairs, see~\cite{Choi03} for instance.
In other words, the number $T(n,i)$ counts the number of involutions of $[n]$ with $n-i$ cycles.
Note that
$b(n,i,i)=\#\{\pi\in\msn: \cpk(\pi)=i, \cyc(\pi)=n-i\}$.
So the following corollary is immediate.
\begin{corollary}
For $0\leq i\leq \lrf{n/2}$,
we have $b(n,i,i)=T(n,i)$.
\end{corollary}

Recall that the numbers $P(n,i)$ satisfy the recurrence relation
\begin{equation}\label{Pni-recu}
P(n,i)=(2i+2)P(n-1,i)+(n-2i)P(n-1,i-1),
\end{equation}
with the initial condition $P(1,0)=1$ and $P(1,i)=0$ for $i\neq 0$ (see~\cite[A008303]{Sloane}).
It follows from~\eqref{bnij} that
\begin{equation}\label{bni-recu}
b(n+1,i,n)=2ib(n,i,n-1)+(n-2i+2)b(n,i-1,n-1).
\end{equation}
Comparing~\eqref{Pni-recu} with~\eqref{bni-recu}, we immediately get the following corollary.
\begin{corollary}
For $0\leq i\leq \lrf{(n-1)/2}$, we have
$b(n+1,i+1,n)=P(n,i)$.
\end{corollary}

Following Dumont~\cite{Dumont96}, if $G_3=\{x\rightarrow xy, y\rightarrow xy\}$,
then for $n\geq 1$, we have
\begin{equation}\label{G3-Dumont}
D_{G_3}^n(x)=x\sum_{\pi\in\msn}x^{\exc(\pi)}y^{n-\exc(\pi)}.
\end{equation}
We end this section by giving the following result.
\begin{proposition}
For $n\geq 1$, we have
\begin{equation}\label{Bnkx-recu}
B_{n+1}^{(k)}(x)=(1+x)B_{n}^{(k)}(x)+x\sum_{i=0}^{n-1}\binom{n}{i}2^{n+1-i}k^{n-i}B_{i}^{(k)}(x)A_{n-i}(x).
\end{equation}
\end{proposition}
\begin{proof}
Let $G$ be the grammar given in Lemma~\ref{DnI-grammar}.
Combining Lemma~\ref{lemma-Bnkx} and Lemma~\ref{lemma-Bnkx-EGF}, we have
$$D_G^n(I)=I\sum_{w\in B_n}x^{\we(w)}y^{n-\we(w)}k^{n-\cyc(w)}.$$
Setting $y=1$, we see that for $n\geq 0$, $D_G^n(I)|_{y=1}=IB_n^{(k)}(x)$.
It follows from~\eqref{G3-Dumont} that for $n\geq 1$,
\begin{equation}
D_G^n(x+y)=2^{n+1}k^nx\sum_{\pi\in\msn}x^{\exc(\pi)}y^{n-\exc(\pi)}.
\end{equation}
By the Leibniz formula, we find that
\begin{align*}
D_G^{n+1}(I)&=\sum_{i=0}^n\binom{n}{i}D_G^i(I)D_G^{n-i}(x+y)\\
&=(x+y)D_G^{n}(I)+\sum_{i=0}^{n-1}\binom{n}{i}D_G^i(I)D_G^{n-i}(x+y).
\end{align*}
Setting $y=1$, we arrive at~\eqref{Bnkx-recu}.
\end{proof}

\section{The $1/k$-derangement polynomials of type $B$}\label{Section03}
Let $p(x)=\sum_{i=0}^dp_ix^i$. We say that $p(x)$ is {\it alternatingly increasing} if
$$p_0\leq p_d\leq p_1\leq p_{d-1}\leq\cdots \leq p_{\lrf{\frac{d+1}{2}}}.$$
We now recall an elementary result.
\begin{proposition}[{\cite{Beck2010}}]\label{prop01}
Let $p(x)$ be a polynomial of degree $d$.
There is a unique symmetric decomposition $p(x)= a(x)+xb(x)$, where $a(x)$ and $b(x)$ are symmetric polynomials
satisfying $a(x)=x^d a(\frac{1}{x})$ and $b(x)=x^{d-1}b(\frac{1}{x})$.
\end{proposition}
It is routine to verify that
\begin{equation}\label{ax-bx-prop01}
a(x)=\frac{p(x)-x^{d+1}p(1/x)}{1-x},~b(x)=\frac{x^dp(1/x)-p(x)}{1-x}.
\end{equation}
From~\eqref{ax-bx-prop01}, we see that $\deg a(x)=d$ and $\deg b(x)\leq d-1$.
We call the ordered pair of polynomials $(a(x),b(x))$ the {\it symmetric decomposition} of $p(x)$.
Let $(a(x),b(x))$ be the symmetric decomposition of $p(x)$. If $a(x)$ and $b(x)$ are both $\gamma$-positive, then we say that
$p(x)$ is {\it bi-$\gamma$-positive}.
As pointed out by Br\"and\'en and Solus~\cite{Branden18},
the polynomial $p(z)$ is alternatingly increasing if and only if the pair of polynomials in its symmetric decomposition are both unimodal
and have nonnegative coefficients. Thus the bi-$\gamma$-positivity of $p(x)$ implies that $p(x)$ is alternatingly increasing.

A permutation $\pi\in\msn$ is a {\it derangement} if it has no fixed points, i.e., $\pi(i)\neq i$ for all $i\in [n]$.
Let $\mdn_n$ be the set of derangements in $\msn$.
Let $d_n(x)=\sum_{\pi\in\mdn_n}x^{\exc(\pi)}$ be the derangement polynomials.
It is well known that the generating function of $d_n(x)$ is given as follows (see~\cite[Proposition~6]{Brenti90}):
\begin{equation}\label{dxz-EGF}
d(x,z)=\sum_{n\geq 0}d_n(x)\frac{z^n}{n!}=\frac{(1-x)}{e^{xz}-xe^{z}}.
\end{equation}
By using continued fractions, Shin and Zeng~\cite[Theorem~11]{Zeng12} obtained the following result.
\begin{theorem}\label{thm-zeng}
For $n\geq 2$, we have
$$\sum_{\pi\in\mdn_n}x^{\exc(\pi)}q^{\cyc(\pi)}=\sum_{i=1}^{\lrf{n/2}}c_{n,k}(q)x^k(1+x)^{n-2k},$$
where $c_{n,k}(q)=\sum_{\pi\in \mdn_n(k)}q^{\cyc(\pi)}$ and $\mdn_n(k)$ is the subset of derangements in $\mdn_n$ with exactly $k$ cyclic valleys and without cyclic double descents.
\end{theorem}

In the following, we present a type $B$ analogue of Theorem~\ref{thm-zeng}.
A {\it fixed point} of $w\in B_n$ is an index $i\in [n]$ such that $w(i)=i$.
Let $\fix(w)$ denote the number of fixed points of $w$.
A {\it derangement of type $B$} is a
signed permutation $\pi\in B_n$ with no fixed points.
Let $\mdn_n^B$ be the set of all derangements in $B_n$.
We say that $i\in [n]$ is an {\it excedance} of $w$ if $w(|w(i)|)>w(i)$.
Let $\exc(w)$ be the number of excedances of $w$. Note that $\exc(w)=\we(w)$ for $w\in\mdn_n$.
The {\it type $B$ derangement polynomials} $d_n^B(x)$ are defined by
$$d_n^B(x)=\sum_{\pi\in \mdn_n^B}x^{\exc(\pi)},$$
which have been studied by Chen et al.~\cite{Chen09} in a slightly different form.
According to~\cite[Theorem 3.2]{Chow09}, the generating function of $d_n^B(x)$ is given as follows:
\begin{equation}\label{dxzB-EGF}
\sum_{n=0}^\infty d_n^B(x)\frac{z^n}{n!}=\frac{(1-x)e^z}{e^{2xz}-xe^{2z}}.
\end{equation}
Combining~\eqref{dxz-EGF} and~\eqref{dxzB-EGF}, we obtain
$$d_n^B(x)=\sum_{i=0}^n\binom{n}{i}2^id_i(x).$$

The type $B$ {\it $1/k$-derangement polynomials} $d_n^B(x;k)$ are defined by
\begin{equation}\label{dxzBk-EGF}
C(x,z;k)=\sum_{n\geq 0}d_n^B(x;k)\frac{z^n}{n!}=\left(\frac{(1-x)e^{kz}}{e^{2kxz}-xe^{2kz}}\right)^{1/k}.
\end{equation}
In particular, $d_0^B(x;k)=d_1^B(x;k)=1,d_2^B(x;k)=1+4kx$ and
$d_3^B(x;k)=1+(12k+8k^2)x+8k^2x^2$.
Comparing~\eqref{Bnkx-deff} with~\eqref{dxzBk-EGF}, we have
$$B_n^{(k)}(x)=\sum_{i=0}^n\binom{n}{i}d_i^B(x;k)x^{n-i}.$$

A grammatical interpretation of the polynomial $d_n^B(x)$ is given by~\cite[Theorem~11]{Ma1801}.
We now give a refinement of Lemma~\ref{lemma-Bnkx}.
\begin{lemma}\label{derangment-grammar01}
If $G_4=\{I\rightarrow I(y+u),x\rightarrow 2kxy,y\rightarrow 2kxy,u\rightarrow 2kxy\}$,
then
\begin{equation}\label{derangment-grammar022}
D_{G_4}^n(I)=I\sum_{w\in B_n}x^{\exc(w)}y^{\aexc(w)}u^{\fix(w)}k^{n-\cyc(w)}.
\end{equation}
\end{lemma}
\begin{proof}
We now introduce a grammatical labeling of $w\in B_n$ as follows:
\begin{itemize}
  \item [\rm ($L_1$)]If $i$ is an excedance, then put a superscript label $x$ right after $i$;
 \item [\rm ($L_2$)]If $i$ is an anti-excedance, then put a superscript label $y$ right after $i$;
\item [\rm ($L_3$)]If $i$ is a fixed point, then put a superscript label $u$ right after $i$;
\item [\rm ($L_4$)]Put a subscript label $I$ right after $w$;
\item [\rm ($L_5$)] Put a subscript label $k$ just before every element of $w$ except the first element in each cycle;
\end{itemize}
Note that the weight of $w$ is given by $Ix^{\exc(w)}y^{\aexc(w)}u^{\fix(w)}k^{n-\cyc(\pi)}$.
For $n=1$, we have $B_1=\{(1^u)_I,(\overline{1}^y)_I,\}$.
Note that $D_{G_4}(I)=I(y+u)$.
Thus the sum of weights of the elements in $B_1$ is given by $D_{G_4}(I)$.
Hence the result holds for $n=1$.
We proceed by induction on $n$.
Suppose we get all labeled permutations in $B_{n-1}$, where $n\geq 2$. Let
$\widetilde{w}$ be obtained from $w \in B_{n-1}$ by inserting $n$ or $\overline{n}$.
When the inserted $n$ or $\overline{n}$ forms a new cycle, the insertion corresponds to the substitution rule $I\rightarrow I(y+u)$.
For the other cases, the changes of labeling are illustrated as follows:
$$\cdots(i^u)\cdots\mapsto \cdots(n^y_ki^x)\cdots;~~\cdots(i^u)\cdots\mapsto \cdots(\overline{n}^x_ki^y)\cdots;$$
$$\cdots(\cdots w(i)^x_k w(i+1)\cdots w(j))\cdots\mapsto \cdots(n^y_k w(i+1)\cdots w(j)_k\cdots w(i)^x)\cdots;$$
$$\cdots(\cdots w(i)^x_k w(i+1)\cdots w(j))\cdots\mapsto \cdots(\overline{n}^x_k w(i+1)\cdots w(j)_k\cdots w(i)^y)\cdots,$$
$$\cdots(\overline{i}^y)\cdots\mapsto \cdots(n^y_k\overline{i}^x)\cdots;~~\cdots(\overline{i}^y)\cdots\mapsto \cdots(\overline{n}^x_k\overline{i}^y)\cdots;$$
$$\cdots(\cdots w(i)^y_k w(i+1)\cdots w(j))\cdots\mapsto \cdots(n^y_k w(i+1)\cdots w(j)_k\cdots w(i)^x)\cdots;$$
$$\cdots(\cdots w(i)^y_k w(i+1)\cdots w(j))\cdots\mapsto \cdots(\overline{n}^x_k w(i+1)\cdots w(j)_k\cdots w(i)^y)\cdots.$$
In each case, the insertion of $n$ or $\overline{n}$ corresponds to one substitution rule in ${G_4}$. By induction,
it is routine to check that the action of $D_{G_4}$ on elements of $B_{n-1}$ generates all elements of $B_{n}$.
This completes the proof.
\end{proof}

We define $$H_{n}(x,y,u;k)=\sum_{w\in B_n}x^{\exc(w)}y^{\aexc(w)}u^{\fix(w)}k^{n-\cyc(w)},$$
$$H:=H(x,y,u,z;k)=\sum_{n=0}^{\infty}H_{n}(x,y,u;k)\frac{z^n}{n!}.$$
\begin{lemma}\label{lemma-Bnkx-EGF}
We have
\begin{equation}\label{Hxyuz}
H(x,y,u,z;k)=\left(\frac{(y-x)e^{kz(y+u-2x)}}{y-xe^{2kz(y-x)}}\right)^{1/k}.
\end{equation}
\end{lemma}
\begin{proof}
Since $D_{G_4}^{n+1}(I)=D_{G_4}(IH_{n}(x,y,u;k))$, it follows that
\begin{align*}
D_{G_4}^{n+1}(I)
&=I(y+u)H_{n}(x,y,u;k)+2kxyI\left(\frac{\partial}{\partial x}+\frac{\partial}{\partial y}+\frac{\partial}{\partial u}\right)H_{n}(x,y,u;k).
\end{align*}
Thus $H_{n+1}(x,y,u;k)=(y+u)H_{n}(x,y,u;k)+2kxy\left(\frac{\partial}{\partial x}+\frac{\partial}{\partial y}+\frac{\partial}{\partial u}\right)H_{n}(x,y,u;k)$.
By rewriting this recurrence relation in terms of the generating function $H$, we have
\begin{equation}\label{recu-Hnkxy03}
\frac{\partial}{\partial z}H=(y+u)H+2kxy\left(\frac{\partial}{\partial x}+\frac{\partial}{\partial y}+\frac{\partial}{\partial u}\right)H.
\end{equation}
It is routine to check that the generating function $$\widetilde{H}(x,y,u,z;k)=\left(\frac{(y-x)e^{kz(y+u-2x)}}{y-xe^{2kz(y-x)}}\right)^{1/k}$$
satisfy~\eqref{recu-Hnkxy03}. Note that $\widetilde{H}(x,y,u,0;k)=1, \widetilde{H}(0,y,u,z;k)=e^{u+y}$ and $\widetilde{H}(x,0,u,z;k)=e^{uz}$.
Hence $\widetilde{H}(x,y,u,z;k)=H(x,y,u,z;k)$.
\end{proof}

We now present a derangement counterpart of Theorem~\ref{thm1}.
\begin{theorem}\label{mainthm02}
For $n\geq 1$, we have $$d_n^B(x;k)=\sum_{\pi\in \mdn_n^B}x^{\we(\pi)}k^{n-\cyc(\pi)}.$$
When $k\geq 1/2$, the polynomials $d_n^B(x;k)$ are bi-$\gamma$-positive. More precisely,
we have
$$d_n^B(x;k)=\sum_{j=0}^{\lrf{(n-1)/2}}p(n,j;k)x^j(1+x)^{n-1-2j}+\sum_{j=0}^{\lrf{n/2}}q(n,j;k)x^j(1+x)^{n-2j},$$
where the numbers $p(n,j;k)$ and $q(n,j;k)$ satisfy the following recurrence system:
\begin{align*}
p(n+1,j;k)&=(1+2kj)p(n,j;k)+4k(n-2j+1)p(n,j-1;k)+\\& 2knp(n-1,j-1;k)+q(n,j;k),\\
q(n+1,j;k)&=2kjq(n,j;k)+4k(n-2j+2)q(n,j-1;k)+2knq(n-1,j-1;k)+\\& (2k-1)p(n,j-1;k),
\end{align*}
with $q(0,0;k)=1$ and $q(0,j;k)=0$ for $j\neq 0$, $p(0,j;k)=0$ for any $j$.
\end{theorem}
\begin{proof}
Comparing~\eqref{dxzBk-EGF} and~\eqref{Hxyuz}, we see that $C(x,z;k)=H(x,1,0,z;k)$,
which leads to the combinatorial interpretation of $d_n^B(x;k)$.
We now consider a change of the grammar given in Lemma~\ref{derangment-grammar01}.
Note that $D_{G_4}(I)=Iy+Iu$ and
$$D_{G_4}(Iy)=I(y^2+yu+2kxy)=Iy(x+y)+Iyu+(2k-1)Ixy.$$
Setting $a=xy,~b=x+y$ and $c=Iy$, we get $D_{G_4}(a)=2kab,D_{G_4}(b)=4ka$,
$$D_{G_4}(I)=c+uI,D_{G_4}(c)=(b+u)c+(2k-1)aI,D_{G_4}(u)=2ka.$$
Consider the grammar
$$G_5=\{I\rightarrow c+uI,c\rightarrow (b+u)c+(2k-1)aI,u\rightarrow 2ka,a\rightarrow 2kab,b\rightarrow 4ka\}.$$
Note that $D_{G_5}(I)=c+Iu,~D^2_{G_5}(I)=(b+2u)c+(u^2+(4k-1)a)I$.
By induction, it is routine to verify that there exist nonnegative integers such that
\begin{equation}\label{Dg5}
D_{G_5}^n(I)=\sum_{i=0}^nu^i\left(\sum_{j=0}^{(n-1-i)/2}p(n,i,j;k)a^jb^{n-1-i-2j}c+\sum_{j=0}^{(n-i)/2}q(n,i,j;k)a^jb^{n-i-2j}I\right).
\end{equation}
Extracting the coefficients of $a^jb^{n-2j}c$ and $a^jb^{n+1-2j}I$ on both sides of the following expansion
$$D_{G_5}^{n+1}(I)=D_{G_5}\left(\sum_{i,j}p(n,i,j;k)u^ia^jb^{n-1-i-2j}c+\sum_{i,j}q(n,i,j;k)u^ia^jb^{n-i-2j}I\right),$$
it is easy to verify the following recurrence system:
\begin{align*}
p(n+1,0,j;k)&=(1+2kj)p(n,0,j;k)+4k(n-2j+1)p(n,0,j-1;k)+\\& 2kp(n,1,j-1;k)+q(n,0,j;k),\\
q(n+1,0,j;k)&=2kjq(n,0,j;k)+4k(n-2j+2)q(n,0,j-1;k)+2kq(n,1,j-1;k)+\\& (2k-1)p(n,0,j-1;k).
\end{align*}
Since $u$ marks fixed points, we have $p(n,1,j-1;k)=np(n-1,0,j-1;k)$ and $q(n,1,j-1;k)=nq(n-1,0,j-1;k)$.
Setting $p(n,0,j;k)=p(n,j;k)$ and $q(n,0,j;k)=q(n,j;k)$, we get the recurrence system of the numbers $p(n,j;k)$ and $q(n,j;k)$.
Setting $u=0$ and $y=1$ in~\eqref{Dg5} and then taking $a=x,~b=1+x$ and $c=I$, we get the symmetric decomposition of $d_n^B(x;k)$.
Clearly, when $k\geq 1/2$, $p(n,j;k)$ and $q(n,j;k)$ are nonnegative, and so $d_n^B(x;k)$ are bi-$\gamma$-positive.
\end{proof}

Define
$$P_n(x;k)=\sum_{j=0}^{\lrf{(n-1)/2}}p(n,j;k)x^j(1+x)^{n-1-2j},$$
$$Q_n(x;k)=\sum_{j=0}^{\lrf{n/2}}q(n,j;k)x^j(1+x)^{n-2j}.$$
In particular, $P_1(x;k)=1,P_2(x;k)=1+x,P_3(x;k)=1+(1+12k)x+x^2,Q_1(x;k)=0,Q_2(x;k)=(4k-1)x$ and $Q_3(x;k)=(8k^2-1)x(1+x)$.
\begin{corollary}\label{cor-derangement}
The polynomials $P_n(x;k)$ and $Q_n(x;k)$ satisfy the recurrence system
\begin{align*}
P_{n+1}(x;k)&=(1+(2kn-2k+1)x)P_n(x;k)+2kx(1-x)P_n'(x;k)+\\&2knxP_{n-1}(x;k)+Q_n(x;k),\\
Q_{n+1}(x;k)&=2knxQ_n(x;k)+2kx(1-x)Q_n'(x;k)+2knxQ_{n-1}(x;k)+\\&(2k-1)xP_n(x;k),
\end{align*}
with initial conditions $P_0(x;k)=0$ and $Q_0(x;k)=1$.
\end{corollary}
\begin{proof}
Define $$p_n(x)=\sum_{j=0}^{\lrf{(n-1)/2}}p(n,j;k)x^j,~q_n(x)=\sum_{j=0}^{\lrf{n/2}}q(n,j;k)x^j.$$
Multiplying both sides of the recurrence system of the numbers $p(n,j;k)$ and $q(n,j;k)$ and summing over all $j$, we get the following recurrence system:
\begin{align*}
p_{n+1}(x)&=(1+4k(n-1)x)p_n(x)+2kx(1-4x)p_n'(x)+2knxp_{n-1}(x)+q_n(x),\\
q_{n+1}(x)&=4knxq_n(x)+2kx(1-4x)q_n'(x)+2knxq_{n-1}(x)+(2k-1)xp_n(x),
\end{align*}
with the initial conditions $p_0(x)=0,p_1(x)=1$, $q_0(x)=1$ and $q_1(x)=0$.
Note that
$$P_n(x;k)=(1+x)^{n-1}p_n\left(\frac{x}{(1+x)^2}\right),$$
$$Q_n(x;k)=(1+x)^{n}q_n\left(\frac{x}{(1+x)^2}\right).$$
Substituting $x\rightarrow {x}/{(1+x)^2}$ into the recurrence system of the polynomials $p_n(x)$ and $q_n(x)$ and
simplifying some terms leads to the desired result.
\end{proof}

A {\it succession} of $\pi\in\msn$ is an index $i$ such that $\pi(i+1)=\pi(i)+1$, where $i\in [n-1]$.
Let $\msn^{s}$ denote the set of permutations in $\msn$ with no successions.
We can now give the following result.
\begin{theorem}\label{theorem-derange}
For $n\geq 1$, we have
$$d_n^B(x;1/2)=\frac{1}{x}d_{n+1}(x)+d_n(x),$$
where $d_n(x)$ is the classical derangement polynomial.
Moreover, we have
\begin{equation}\label{dbB12}
d_n^B(x;1/2)=\sum_{\pi\in\ms_{n+1}^s}x^{\asc(\pi)}.
\end{equation}
\end{theorem}
\begin{proof}
Let $P_n(x)=P_n(x;1/2)$ and $Q_n(x)=Q_n(x;1/2)$.
It follows from Theorem~\ref{mainthm02} that
$d_n^B(x;1/2)=P_n(x)+Q_n(x)$.
By using Corollary~\ref{cor-derangement}, we see that the polynomials $P_n(x)$ and $Q_n(x)$ satisfy the following recurrence system:
\begin{align*}
P_{n+1}(x)&=(1+nx)P_n(x)+x(1-x)P_n'(x)+nxP_{n-1}(x)+Q_n(x),\\
Q_{n+1}(x)&=nxQ_n(x)+x(1-x)Q_n'(x)+nxQ_{n-1}(x),
\end{align*}
with initial conditions $P_0(x)=0,P_1(x)=1,Q_0(x)=1$ and $Q_1(x)=0$.
According to~\cite[Eq.~(3.2)]{Liu07}, the polynomials $Q_n(x)$ satisfy the same recurrence relation and initial conditions as $d_n(x)$, so they agree.
We now prove that $$P_n(x)=\frac{1}{x}d_{n+1}(x).$$
Clearly, it holds for $n=0,1,2$. Assume it holds for $n$. Then we get
\begin{align*}
P_{n+1}(x)&=\frac{1+nx}{x}d_{n+1}(x)+\frac{x(1-x)}{x^2}(xd_{n+1}'(x)-d_{n+1}(x))+\frac{nx}{x}d_n(x)+d_n(x)\\
&=(n+1)d_{n+1}(x)+(1-x)d_{n+1}'(x)+(n+1)d_n(x)\\
&=\frac{1}{x}d_{n+2}(x),
\end{align*}
as desired. The combinatorial interpretation~\eqref{dbB12} follows immediately from~\cite[Eq.~(3.8)]{Roselle68}.
This completes the proof.
\end{proof}

Let $$S_n(x)=\sum_{\pi\in\ms_{n+1}^s}x^{\asc(\pi)}.$$
The following generating function is a restatement of~\eqref{dbB12}:
\begin{equation}\label{Snx-EGF}
\sum_{n=0}^\infty S_n(x)\frac{z^n}{n!}=e^z\left(\frac{1-x}{e^{xz}-xe^z}\right)^2.
\end{equation}
Combining~\eqref{Ankx-deff},~\eqref{dxzB-EGF} and~\eqref{Snx-EGF}, we get the following result.
\begin{corollary}\label{SnxdiBx}
For $n\geq 0$, we have $$S_n(x)=\frac{1}{2^n}\sum_{i=0}^n\binom{n}{i}d_i^B(x)d_{n-i}^B(x),$$
$$S_n(x)=\sum_{i=0}^n\binom{n}{i}A_i(x)d_{n-i}(x).$$
\end{corollary}
It would be interesting to present a bijective proof of Corollary~\ref{SnxdiBx}.

\bibliographystyle{amsplain}

\begin{thebibliography}{10}
%
%
%



\bibitem{Athanasiadis17}
C.A. Athanasiadis, \textit{Gamma-positivity in combinatorics and geometry}, S\'em. Lothar. Combin., \textbf{77} (2018), Article B77i.


%
\bibitem{Beck2010}
M. Beck, A. Stapledon, \textit{On the log-concavity of Hilbert series of Veronese subrings and Ehrhart series}, Math. Z., \textbf{264} (2010), 195--207.


%


%



\bibitem{Branden18}
P. Br\"and\'{e}n and L. Solus, \textit{Symmetric decompositions and real-rootedness}, Int Math. Res Notices, rnz059 (2019), https://doi.org/10.1093/imrn/rnz059.


\bibitem{Brenti90}
F. Brenti, \textit{Unimodal polynomials arising from symmetric functions}, Proc. Amer. Math. Soc., \textbf{108} (1990), 1133--1141.

\bibitem{Brenti94}
F. Brenti, \textit{$q$-Eulerian polynomials arising from Coxeter groups}, European J. Combin.,
\textbf{15} (1994), 417--441.

\bibitem{Brenti00}
F. Brenti, \textit{A class of $q$-symmetric functions arising from plethysm},
J. Combin. Theory Ser. A, \textbf{91} (2000), 137--170.





\bibitem{Chao19}
T.-W. Chao, J. Ma, S.-M. Ma, Y.-N. Yeh, \textit{$1/k$-Eulerian polynomials and $k$-inversion sequences}, Electr. J. Comb. \textbf{26}(3) (2019), P3.35.

%
\bibitem{Chen09}
W.Y.C. Chen, R.L. Tang and A.F.Y. Zhao, \textit{Derangement polynomials and excedances of type $B$},
Electron. J. Combin., \textbf{16}(2) (2009), Research Paper 15.



\bibitem{Chen17}
W.Y.C. Chen, A.M. Fu, \textit{Context-free grammars for permutations and increasing trees}, Adv. in Appl. Math., \textbf{82} (2017), 58--82.


\bibitem{Choi03}
J.Y. Choi and J.D.H. Smith, \textit{On the unimodality and combinatorics of Bessel numbers}, Discrete Math., \textbf{264} (2003), 45--53.

\bibitem{Chow09}
C.-O. Chow, \textit{On derangement polynomials of type $B$}, II, J. Combin. Theory Ser. A, \textbf{116} (2009),
816--830.




%

\bibitem{Dumont96}
D. Dumont, \textit{Grammaires de William Chen et d\'erivations dans les arbres et
arborescences}, S\'em. Lothar. Combin., \textbf{37}, Art. B37a (1996), 1--21.



\bibitem{Foata70}
D. Foata and M. P. Sch\"utzenberger, \textit{Th\'eorie g\'eometrique des polyn\^omes eul\'eriens}, Lecture Notes in Math., vol. 138, Springer, Berlin, 1970.

\bibitem{Foata73}
D. Foata and V. Strehl, \textit{Euler numbers and variations of permutations}, in Colloquio Internazionale
sulle Teorie Combinatorie (Roma, 1973), Tomo I, Atti dei Convegni Lincei, No. \textbf{17}, Accad. Naz.
Lincei, Rome, 1976, pp. 119--131.
%
%
%





\bibitem{Gal05}
S.R. Gal, \textit{Real root conjecture fails for five and higher-dimensional spheres}, Discrete Comput. Geom., \textbf{34} (2005), 269--284.
%
\bibitem{Gessel78}
I. Gessel and R.P. Stanley, \textit{Stirling polynomials}, J. Combin. Theory Ser.
A, \textbf{24} (1978), 25--33.

%
%


%

\bibitem{Lin15}
Z. Lin, J. Zeng, \textit{The $\gamma$-positivity of basic Eulerian polynomials via group actions}, J. Combin. Theory Ser. A, \textbf{135} (2015), 112--129.
%
%
\bibitem{Liu07}
Lily L. Liu, Yi Wang, \textit{A unified approach to polynomial sequences with only real zeros}, Adv. in Appl. Math., \textbf{38} (2007), 542--560.

%
\bibitem{Ma15}
S.-M. Ma, T. Mansour, \textit{The $1/k$-Eulerian polynomials and $k$-Stirling permutations}, Discrete Math., \textbf{338} (2015), 1468--1472.

\bibitem{Ma1801}
S.-M. Ma, J. Ma, Y.-N. Yeh, B.-X. Zhu, \textit{Context-free grammars for several polynomials
associated with Eulerian polynomials}, Electron. J. Combin.,
\textbf{25}(1) (2018), \#P1.31.
%
\bibitem{Ma19}
S.-M. Ma, J. Ma, Y.-N. Yeh, \textit{$\gamma$-positivity and partial $\gamma$-positivity of descent-type polynomials}, J. Combin. Theory Ser. A, \textbf{167} (2019), 257--293.
%


\bibitem{Petersen07}
T.K. Petersen, \textit{Enriched $P$-partitions and peak algebras}, Adv. Math., \textbf{209}(2) (2007), 561--610.


\bibitem{Roselle68}
D.P. Roselle, \textit{Permutations by number of rises and successions}, \newblock{\em Proc. Amer. Math. Soc.}, \textbf{19} (1968), 8--16.

\bibitem{Savage12}
C.D. Savage and G. Viswanathan, \textit{The $1/k$-Eulerian polynomials}, Electron J. Combin.,
\textbf{19} (2012), \#P9.
%




\bibitem{Zeng12}
H. Shin, J. Zeng, \textit{The symmetric and unimodal expansion of Eulerian polynomials via continued fractions}, European J. Combin., \textbf{33} (2012), 111--127.
%
\bibitem{Zeng16}
H. Shin and J. Zeng, \textit{Symmetric unimodal expansions of excedances in colored permutations}, European
J. Combin., \textbf{52} (2016), 174--196.
\bibitem{Sloane}
N.J.A. Sloane, The On-Line Encyclopedia of Integer Sequences,
published electronically at http://oeis.org, 2010.


\bibitem{Stembridge97}
J. Stembridge, \textit{Enriched P-partitions}, Trans. Amer. Math. Soc., \textbf{349}(2): (1997) 763--788.
%


\bibitem{Zhuang17}
Y. Zhuang, \textit{Eulerian polynomials and descent statistics}, Adv. in Appl. Math., \textbf{90} (2017), 86--144.
\end{thebibliography}

\end{document}